\documentclass[11pt]{article}

\title{Kneser's theorem for codes and $\ell$-divisible set families}
\author{Chenying Lin\thanks{Fakult\"at f\"ur Mathematik, Universit\"at Regensburg. \texttt{Chenying.Lin@mathematik.uni-regensburg.de}}, Gilles Z{\'e}mor\thanks{Institut de Math\'ematiques de Bordeaux, UMR 5251. \texttt{zemor@math.u-bordeaux.fr}}\;\thanks{Institut universitaire de France}}
\date{\today}

\usepackage[T1]{fontenc}

\usepackage[a4paper,hmargin=2.8cm,vmargin=2.0cm,includeheadfoot]{geometry}
\usepackage{float}
\usepackage[english]{babel}
\usepackage{amsmath}
\usepackage{amssymb}
\usepackage{dsfont}
\usepackage{array}
\usepackage{latexsym}
\usepackage{etoolbox}
\usepackage{mathtools}
\usepackage{underscore}

\usepackage{enumitem}

\usepackage{amsthm}
\theoremstyle{plain}

\usepackage{url}
\usepackage[dvipsnames]{xcolor}
\usepackage[bookmarks=false]{hyperref}
\hypersetup{colorlinks=true,citecolor=ForestGreen,linkcolor=blue,urlcolor=magenta}
\usepackage[capitalize,nameinlink,noabbrev]{cleveref}
\usepackage{tikz}
\usepackage{bm}

\newtheorem{thm}{Theorem}[section]
\newtheorem{prop}[thm]{Proposition}

\newtheorem{lem}[thm]{Lemma}

\theoremstyle{definition}
\newtheorem{defn}[thm]{Definition}
\newtheorem{rmk}[thm]{Remark}

\newcommand{\Z}{\mathbb{Z}} 
\newcommand{\N}{\mathbb{N}} 

\newcommand{\F}{\mathbb F}
\newcommand{\cF}{\mathcal F}
\newcommand{\one}{\mathbf 1}
\def\spn#1{\langle #1\rangle}
\def\inner#1#2{\langle #1 , #2\rangle}
\DeclareMathOperator{\supp}{Supp}
\DeclareMathOperator{\St}{St}

\newcommand{\blambda}{\bm\lambda}

\begin{document}

\maketitle

\begin{abstract}
    A $k$-wise $\ell$-divisible set family is a collection $\cF$ of subsets of
${ \{1,\ldots,n \} }$ such that any intersection of $k$ sets in $\cF$ has
cardinality divisible by $\ell$. If $k=\ell=2$, it is well-known that $|\cF|\leq
2^{\lfloor n/2 \rfloor}$. We generalise this by proving that $|\cF|\leq
2^{\lfloor n/p\rfloor}$ if $k=\ell=p$, for any prime number $p$.
    
    For arbitrary values of $\ell$, we prove that $4\ell^2$-wise
$\ell$-divisible set families $\cF$ satisfy $|\cF|\leq 2^{\lfloor
n/\ell\rfloor}$ and that the only families achieving the upper bound are atomic,
meaning that they consist of all the unions of disjoint subsets of size $\ell$. 
This improves upon a recent result by Gishboliner, Sudakov and Timon,
that arrived at the same conclusion for $k$-wise $\ell$-divisible families, with
values of $k$ that behave exponentially in $\ell$.

    Our techniques rely heavily upon a coding-theory analogue of Kneser's Theorem from
additive combinatorics.
\end{abstract}

\section{Introduction}
\subsection{Results and context}
We are interested in the maximal size of a family $\cF\subset 2^{[n]}$ of subsets of
$[n]:=\{1,2,\ldots,n\}$, such that the intersection of any $k$ subsets of $\cF$
satisfies some divisibility properties, where $k$ is some integer. 
More specifically, let us say that $\mathcal{F}\subset 2^{[n]}$ is~ {\em
$k$-wise $\ell$-divisible} if $|A_1\cap\cdots\cap A_k|$ is divisible by $\ell$ for any $A_1,\ldots,A_k\in\mathcal{F}$.
Our main result is the following:
\begin{thm}\label{MainThmGeneralizedPEventown}
Let $p$ be a prime integer and let $\cF\subset 2^{[n]}$ 
be a $p$-wise $p$-divisible set family.
Then, $|\mathcal{F}|\leq 2^{\lfloor n/p\rfloor}$.
\end{thm}
We note that the upper bound $2^{\lfloor n/p\rfloor}$ is the best possible,
since it can be achieved by {\em atomic} families. Let us say that a set family $\mathcal{F}$ is
atomic if it consists of all the unions of some pairwise disjoint subsets of
$[n]$ called the \textit{atoms} of $\cF$. By choosing atoms of cardinality $p$, we see
that the corresponding atomic family $\cF$ is $k$-wise $p$-divisible for any $k$, and 
achieves the upper bound of \Cref{MainThmGeneralizedPEventown}.

For $p=2$, \Cref{MainThmGeneralizedPEventown} recovers a folklore
result sometimes called the {\em Eventown} Theorem~\cite{babai}.
For $p=3$, \Cref{MainThmGeneralizedPEventown} is best possible in the
following sense: as pointed out by Frankl and Odlyzko \cite{FranklOdlyzko}, from
a Hadamard matrix of order $12$ one can derive a family $\cF$ of subsets of $[12]$ of size
$|\cF|=24>2^{12/3}$ such that the intersection of any two subsets of $\cF$ has
cardinality divisible by $3$: and more generally, for $n=12m$, one obtains such
families of size $24^m$.
Therefore, the hypothesis in \Cref{MainThmGeneralizedPEventown} that
$\cF$ be $3$-wise $3$-divisible cannot be weakened to being only $2$-wise
$3$-divisible.

For $p=2$, maximal $2$-wise $2$-divisible families need not be atomic with atoms of size $2$: indeed,
any binary self-dual code, whose structure can be very intricate, consists of a set
of characteristic vectors of a maximal $2$-wise $2$-divisible family. However,
it was shown in \cite{kwise4} that maximal $3$-wise $2$-divisible families are
atomic. The following companion result to \Cref{MainThmGeneralizedPEventown}
provides a generalisation.

\begin{thm}\label{MainThmPEventownExtremal}
    Let $p$ be a prime integer and let $\mathcal{F}\subset 2^{[n]}$ be a
$(p+1)$-wise $p$-divisible set family. If $|\mathcal{F}|>2^{\lfloor n/p\rfloor
-1}$, then $\mathcal{F}$ is contained in an atomic family with atoms of size $p$.
\end{thm}

\Cref{MainThmGeneralizedPEventown} and \Cref{MainThmPEventownExtremal}
are inspired by the recent work of Gishboliner, Sudakov and Tomon
\cite{Gishboliner2022Small}. For every positive integer $\ell$, they showed that
there exists an integer $k$ that is a function of $\ell$ with the following property. For
every $k$-wise $\ell$-divisible set family $\mathcal{F}\subset 2^{[n]}$, the cardinality $|\mathcal{F}|$ 
is not greater than $2^{\lfloor n/\ell\rfloor}$. Moreover, if $|\mathcal{F}|$ is large enough, 
then $\mathcal{F}$ is a subfamily of an atomic family with atoms of cardinality $\ell$. The guaranteed upper bound on $k$ given in
\cite{Gishboliner2022Small} is exponential in $\ell$. Summarising:

\begin{thm}[Theorem 9, \cite{Gishboliner2022Small}]\label{SudakovStructureThm}
    Given $\ell>0$, there exists a positive integer
$k=2^{O(\ell\log\ell)}$ with the following property. Let 
$\mathcal{F}\subset 2^{[n]}$ be $k$-wise $\ell$-divisible. Then
    \begin{itemize}
        \item $|\mathcal{F}|\leq 2^{\lfloor n/\ell\rfloor}$;
        \item if $|\mathcal{F}|>2^{\lfloor n/\ell\rfloor -1}$, then
$\mathcal{F}$ is contained in an atomic family with atoms of size $\ell$.
    \end{itemize}
\end{thm}

\Cref{SudakovStructureThm}
essentially solved a conjecture of Frankl and Odlyzko \cite{FranklOdlyzko}.
\cref{MainThmGeneralizedPEventown} and \cref{MainThmPEventownExtremal}
go some way towards finding the optimal value of $k$ in
\Cref{SudakovStructureThm}. Though \Cref{MainThmGeneralizedPEventown} and \Cref{MainThmPEventownExtremal} only deal with prime
values of $\ell$, we also obtain a result for general $\ell$ that reduces the
value of $k$ from exponential in $\ell$ to polynomial in $\ell$. Specifically:

\begin{thm}[structure theorem for $\ell$-divisible set families]\label{MainThmGeneral}
    Let $\ell$ be a positive integer and let $\mathcal{F}\subset 2^{[n]}$ be a $4\ell^2$-wise $\ell$-divisible set family. Then
    \begin{itemize}
        \item $|\mathcal{F}|\leq 2^{\lfloor n/\ell\rfloor}$;
        \item if $|\mathcal{F}|>2^{\lfloor n/\ell\rfloor -1}$, then
$\mathcal{F}$ is contained in an atomic family with atoms of size $\ell$.
    \end{itemize}
\end{thm}

\subsection{Weakening the $k$-wise $\ell$-divisibility hypothesis}
To give the full context of \Cref{SudakovStructureThm} and \Cref{MainThmGeneral},
we should stress that the conjecture of Frankl and Odlyzko and
the main result of \cite{Gishboliner2022Small} weaken the $k$-wise
$\ell$-divisibility hypothesis to only requiring that the intersections of $k$ {\em
distinct} subsets of $\cF$ has cardinality divisible by $\ell$. This issue of weakening
$\ell$-divisibility probably originates in a question of Erd\"os who asked
what becomes of the Eventown Theorem if we only require the intersection of
distinct sets to have even size. This last question was solved independently by Berlekamp \cite{Eventown2} and Graver \cite{Eventown1}. They showed that $|\mathcal{F}|\leq 2^{n/2}$ if $n\geq 6$ and $n$ is even; and~ $|\mathcal{F}|\leq 2^{(n-1)/2}+1$ if $n\geq 7$ and $n$ is odd.
The main result of \cite{Gishboliner2022Small} stated in full, reads:
\begin{thm}[Theorem 1, \cite{Gishboliner2022Small}]\label{SudakovMainResult}
    Given $\ell\in\N_{>0}$, there exists a positive integer $k=k(\ell)$ with the following property. Let $\mathcal{F}\subset 2^{[n]}$ be a set family such that $|A_1\cap\cdots\cap A_k|$ is divisible by $\ell$ for all pairwise-distinct $A_1,\ldots,A_k\in\mathcal{F}$. Then 
    \[
    |\mathcal{F}|\leq 2^{\lfloor n/\ell\rfloor}+c\quad\text{for a constant }c=c(\ell,k).
    \]
    If $n$ is large enough and $\ell$ divides $n$, then $c=0$.
\end{thm}
However, it is shown in \cite[Section 4]{Gishboliner2022Small} that any theorem that
follows the format of \Cref{SudakovStructureThm}
or of \Cref{MainThmGeneral} can be transformed into a theorem that
yields the stronger statement of \Cref{SudakovMainResult}, without
modifying the value of $k$. 
Therefore, in the present paper we focus on the statements of \Cref{MainThmGeneralizedPEventown}, 
\Cref{MainThmPEventownExtremal} and \Cref{MainThmGeneral} without
dwelling on refinements that weaken the $k$-wise $\ell$-divisibility hypothesis.

\subsection{Methods}
We can think of a family $\cF\subset 2^{[n]}$ as a set of functions
$[n]\rightarrow\{0,1\}$, or equivalently as binary $n$-tuples that can be
embedded in any vector space $\F^n$ for any field $\F$. 
Intersecting two subsets $A,B$ corresponds therefore to taking the
product of the associated functions, or to taking the coordinate-wise product of
the corresponding vectors in $\F^n$. If $v=(v(1),v(2),\ldots,v(n))$ and
$w=(w(1),\ldots, w(n))$ are two vectors of $\F^n$, let us therefore denote
by $v*w$ the vector $(v(1)w(1),\ldots v(n)w(n))$. If $C$ and $D$ are two
vector subspaces of $\F^n$, we denote by $C*D$ the subspace generated by all
products $c*d$, for $c\in C$ and $d\in D$. To lighten notation, we shall write
$vw$ and $CD$ rather than $v*w$ and $C*D$ when no confusion should arise. 
The product $CD$ of spaces $C$ and $D$ has been called the Hadamard product, the
Schur product, the star product and we shall refer to it here simply as
``product''. Products of spaces have found numerous applications in Coding
theory and related fields, see~\cite{randriam15} for an extensive survey.
When taking the product of a space $C$ with itself, following \cite{randriam15}
we shall write $C^{\spn{2}}=C*C,\ldots, C^{\spn{i}}=C^{\spn{i-1}}*C,\ldots$.

A crucial observation, already present in \cite{Gishboliner2022Small}, is that
if $\cF$ is $k$-wise $p$-divisible for some prime $p$, then the functions of $\cF$
generate a sub-vector space $V$ in $\F_p ^n$, where $\F_p$ is the finite
field of size $p$, that must satisfy:
\begin{equation}\label{eq:one}
\one\in (V^{\spn{k}})^\perp.
\end{equation}
In \eqref{eq:one}, the vector $\one$ denotes the all-one vector and orthogonality is relative to
the standard inner product. Our proof strategy is therefore to study the
sequence
\[
V,V^{\spn{2}},\ldots ,V^{\spn{k}}.
\]
If this sequence grows too quickly, then it will eventually fill up the whole
space $\F_p ^n$ and contradict \eqref{eq:one}. To analyse what happens when the
sequence grows sufficiently slowly to allow \eqref{eq:one}, our main tool will be
{\em Kneser's Theorem for codes}~\cite[Theorem 3.3]{KneserCode} (see also \cite{beck2017}). It states
 that for any two non-zero codes (vector spaces) $C,D\subset\mathbb{F}^n$, we have
\[
\operatorname{dim}CD\geq\operatorname{dim}C+\operatorname{dim}D-\operatorname{dim}\operatorname{St}(CD),
\]
where $\operatorname{St}(CD)=\{x\in \mathbb{F}^n:xCD\subset CD\}$.
Kneser's Theorem for codes is named after Kneser's original Theorem on Abelian
groups \cite{kneser1953}, which is often used in additive combinatorics. The
version we use here really is about
vector spaces, but was first proved in a Coding Theory context, hence the
reference to codes used here as shorthand for vector space. Our proof strategy
will consist in trying to show that the stabiliser of $V^{\spn{i}}$ must grow with
$i$ until eventually the dimension of the stabiliser of $V^{\spn{k}}$ must
equal the dimension of $V^{\spn{k}}$. When this happens, $\cF$ must be included in an
atomic family with $p$-divisible atoms.

\medskip

The paper is organised as follows: \Cref{section:Kneser's theorem for
codes} will give some background on Kneser's Theorem for codes.
\Cref{section: Generalized Eventown Theorem for prime-divisible set
families} is devoted to proving \Cref{MainThmGeneralizedPEventown}
and \Cref{MainThmPEventownExtremal}. 
In \Cref{section:Establishing the atomicity of k-closed families}, we prove  \Cref{MainThmGeneral}.
Finally, \Cref{sec:conclusion} gives some concluding comments.

\section{Kneser's theorem for codes}\label{section:Kneser's theorem for codes}

Let $n\in \Z_{>0}$. For every element $v\in\F^n$, we use $v(i)$ to denote the $i$-th coordinate of $v$. The \textit{support} of $v$ is $\supp(v)=\{i\in[n]:v(i)\neq 0\}$; the \textit{support} of a code (vector space) $C\subset\F^n$ is $\supp(C)=\{i\in[n]:v(i)\neq 0\text{ for some }v\in C\}$. We say that a code $C\subset\F^n$ is of \textit{full-support} if the support of $C$ is $[n]$.

\medskip

Kneser's Theorem for codes will be the central tool in this article. It involves the notion of the {\em stabiliser} of a code.

\begin{defn}[stabiliser]
    Let $\F$ be a field. Let $C\subset\F^n$ be a code. Then the \textit{stabiliser} of $C$ is defined as
    \[
    \operatorname{St}(C)=\{ x\in\F^n:xC\subset C \}.
    \]
\end{defn}

Some comments are in order. First note that $\operatorname{St}(C)\subset\F^n$ is also a code.
Notice also that if $C=C_1\oplus C_2$, where $C_1$ and $C_2$ are codes with disjoint supports, then 
vectors that are constant on the support of $C_1$ and constant on the support of $C_2$ belong to the stabiliser $\St(C)$. We have a converse result: indeed, the stabiliser $\operatorname{St}(C)$ of a code $C$ is not only a code, it is also stable by the product operation $(*)$, which makes it a subalgebra of $\F^n$. From this it is not difficult to prove that the stabiliser is generated by a basis of vectors of disjoint supports and that are constant on their supports (\cite[Lemma 2.7]{KneserCode}. It follows that 
we have the following characterisation of the stabiliser of a code:

\begin{prop}[Lemma 2.10, \cite{KneserCode}]\label{StabiliserDecomposition}
    Let $\F$ be a field. Then any full-support code $C\subset\F^n$ decomposes as
    \[
    C=C_1\oplus\cdots\oplus C_m,
    \]
    where $m=\dim(\operatorname{St}(C))$ and the supports of $C_1,\ldots,C_m$ are non-empty and form a partition of $[n]$. This decomposition is unique and maximal, the $C_i$'s do not decompose into a direct sum of more than one code with disjoint non-zero supports.
\end{prop}

Kneser's Theorem for codes states:

\begin{thm}[Theorem 3.3, \cite{KneserCode}]\label{Kneser}
    Let $\F$ be a field. Let $C,D\subset\F^n$ be two codes. Then
    \[
    \dim CD\geq\dim C+\dim D-\dim\operatorname{St}(CD).
    \]
\end{thm}

Equivalently, \Cref{Kneser} states that if $\dim CD\leq \dim C +\dim D -m$, then $CD$ must decompose into the direct sum of at least $m$ codes with disjoint non-zero supports.

From now on, ``Kneser's Theorem'' will refer to \Cref{Kneser}.

Any code $C$ is stabilised by the scalar multiples of the all-one vector $\one$. When
these are the only stabilisers of $C$, i.e. when $\dim\St(C)=1$, we shall say
that $C$ has {\em trivial stabiliser}. Since we will need to study powers
$C^{\spn{k}}$ of a code $C$, the following straightforward consequence of
Kneser's Theorem will be of use to us. 

\begin{lem}
\label{lem:kneserk}
Let $k\in\Z_{>0}$ and let $\F$ be a field. Let $C\subset\F^n$ be a code such that $C^{\spn{k}}$ has trivial stabiliser. Then $\dim C^{\spn{k}}\geq k\dim C -k+1$.
\end{lem}

\begin{proof}
If $C^{\spn{k}}$ has trivial stabiliser, then so do $C,C^{\spn{2}},\ldots, C^{\spn{k-1}}$. 
Kneser's Theorem \ref{Kneser} implies therefore
\begin{align*}
\dim C^{\spn{k}}&\geq \dim C^{\spn{k-1}} + \dim C -1\\
&\geq \dim C^{\spn{k-2}} + 2(\dim C-1)\\
&\geq\cdots\\
&\geq \dim C + (k-1)(\dim C -1)\\
&=k\dim C -k+1,
\end{align*}
as claimed.
\end{proof}

\section{Prime-divisible set families}
\label{section: Generalized Eventown Theorem for prime-divisible set families}

In this section, we will prove \Cref{MainThmGeneralizedPEventown} and \Cref{MainThmPEventownExtremal}. 

Since we will apply \Cref{StabiliserDecomposition} to vector spaces generated
by set families, we also define the notion of \textit{full-support} for set
families. The support of a set family $\cF\subset 2^{[n]}$ is the union of its
members,
$\supp(\cF)=\bigcup_{F\in\cF}F$. We shall say that $\cF$ is of \textit{full-support} if $\supp(\cF)=[n]$.

\begin{rmk}\label{rmk:assume full support}
Since a family $\cF\subset 2^{[n]}$ can be considered to be defined over its
support, and since $2^{\lfloor n'/p\rfloor}\leq2^{\lfloor n/p\rfloor}$ for
$n'\leq n$,
we observe that it is sufficient to prove \Cref{MainThmGeneralizedPEventown} and \Cref{MainThmPEventownExtremal} for full-support set families. 
\end{rmk}

For the rest of this section, we will assume all set families are of full-support.

As previously mentioned, we shall view the elements of a set family $\cF\subset 2^{[n]}$ as elements of $\{0,1\}^n$.
We shall denote by $\cF^k$ the set of (coordinate-wise) products of $k$, not necessarily distinct, elements of $\cF$.
When the family $\cF$ is $p$-divisible, for some prime $p$, we shall regularly embed $\cF$ in the vector space $\F_p^n$.

Let us denote by $\inner{x}{y}=x_1y_1+\cdots +x_ny_n\in\F_p$ the standard inner product of two vectors 
$x=(x_1,\ldots,x_n),y=(y_1,\ldots,y_n)\in\F_p^n$. 
Let us remark that a family $\cF$ is $k$-wise $p$-divisible if and only if $\inner{f}{\one}=0$ for any vector $f\in \cF^k$. If we denote by $V$ the $\F_p$-vector space generated by $\cF$, we therefore have that $\cF$ is $k$-wise $p$-divisible if and only if $V^{\spn{k}}\subset\one^{\perp}$. We state the following proposition for future reference:

\begin{prop}\label{codes k closed}
  Let $k\in\Z_{>0}$ and let $p$ be a prime integer. The family $\cF\subset 2^{[n]}$ is $k$-wise   $p$-divisible if and only if the $\F_p$-vector space $V$ generated by $\cF$ satisfies $V^{\spn{k}}\subset\one^{\perp}$.
\end{prop}

For a family $\cF\subset\{0,1\}^n$, we will need to relate the cardinality of $\cF$ to the dimension of the $\F_p$-vector space $V$ generated by $\cF$. It is straightforward, e.g. \cite[Theorem 2]{dimV}, to show that 
\begin{equation}
    \label{eq:Odlyzko}
    |V\cap\{0,1\}^n|\leq 2^{\dim V}.
\end{equation}
Indeed, by Gaussian elimination, after possibly permuting coordinates, there exists a basis of $V$ whose vectors make up the rows of a matrix $G$ of the form $G=[I_r | A]$, where $I_r$ is the $r\times r$ identity matrix, $r=\dim V$. (In Coding Theory language, $G$ is a generator matrix of $V$ in systematic form). Therefore, the only linear combinations of the rows of $G$ that yield vectors in $\{0,1\}^n$ must have coefficients in $\{0,1\}$: hence, $|V\cap\{0,1\}^n|\leq 2^{\dim V}$.

The above upper bound cannot be improved in all generality because atomic families $\cF$ achieve it. However, 
we shall prove the following improvement for families that we will be dealing with:

\begin{lem}\label{lem: improved Odlyzo}
    Let $p\geq 3$ be a prime number and let $\cF\subset\{0,1\}^n$ contains at least two non-zero subsets. Let $V$ be $\F_p$-vector space generated by $\cF$, and suppose $\dim(\operatorname{St}(V^{\spn{3}}))=1$. Then $|V\cap\{0,1\}^n|\leq 2^{\dim(V)-1}$.
\end{lem}

The proof of \Cref{lem: improved Odlyzo} is more technical than the rest, and we therefore postpone it to the end of this section.

For the rest of this section, $p$ will be a fixed prime integer, $\cF\subset\{0,1\}^n$ will be a $p$-divisible family, and $V$ will denote the $\F_p$-vector space generated by $\cF$.

We will handle the case when $V$ has non-trivial stabiliser using \Cref{StabiliserDecomposition} and induction on $n$. To deal with the case when $V$ has trivial stabiliser, we have the following lemma.

\begin{lem}
\label{lem:p+1}
Let $t\in \Z_{>0}$. Assume $\cF$ contains at least two non-zero subsets. Suppose that $V^{\spn{t+1}}$ has
trivial stabiliser. Then we have $\dim V>t+1$ and $\dim V^{\spn{t+1}}>(t+1)^2$.
\end{lem}
\begin{proof}
If $\dim V\leq t+1$, then $V$ has a basis $\mathcal{B}$ of $\{0,1\}$-vectors
with $|\mathcal{B}|\leq t+1$.
Since $V^{\spn{k}}$ is generated by all $k$-wise products of vectors of $\mathcal{B}$,
we have that $V^{\spn{k}}=V^{\spn{t+1}}$ for all $k\geq t+1$. In particular
$(V^{\spn{t+1}})^{\spn{2}}=V^{\spn{t+1}}$ which contradicts $V^{\spn{t+1}}$
having trivial stabiliser since $\dim V^{\spn{t+1}}\geq\dim V\geq 2$.
This proves $\dim V>t+1$.

Applying \Cref{lem:kneserk} with $k=t+1$ gives
\[
\dim V^{\spn{t+1}}\geq (t+1)(t+2)-(t+1)+1>(t+1)^2 \qedhere
\]
\end{proof}

It will be useful to consider the restrictions of $\cF$ on some subsets of
$[n]$:

\begin{defn}[restriction of set families]
    For any subset $A\subset [n]$ we define the restriction of $\cF$ to $A$ as
$\cF|_A:=\{F\cap A\,:\,F\in\cF\}$.
\end{defn}

\begin{lem}
\label{lem:C1+C2}
Let $k\in\Z_{>0}$. Let $\cF$ be $k$-wise $p$-divisible. 
Suppose that we have
\[
V^{\spn{k}} = C_1\oplus C_2
\]
where $C_1$ and $C_2$ have disjoint non-zero supports $S_1$ and $S_2$ such that
$[n]=S_1\cup S_2$. Let us define 
\[
\cF_1=\cF|_{S_1}\quad\text{and}\quad\cF_2=\cF|_{S_2}.
\]
Then $\cF_1$ and $\cF_2$ are $k$-wise $p$-divisible. Furthermore we have
$|\cF|\leq |\cF_1||\cF_2|$.
\end{lem}
\begin{proof}
One checks that $C_i$, restricted to its support, can only be equal to the code generated by 
$\cF_i^k$, $i=1,2$. Furthermore, it follows from \Cref{codes k closed} that $V^{\spn{k}}\subset\one^\perp$. Since $C_i\subset V^{\spn{k}}$, the code $C_i$ also satisfies $C_i\subset\one^\perp$. Therefore, by \Cref{codes k closed}, each $\cF_i$ is $k$-wise $p$-divisible.

Finally, note that the map
\begin{eqnarray*}
\cF & \rightarrow & \cF_1\times\cF_2\\
A & \mapsto & (A\cap S_1, A\cap S_2)
\end{eqnarray*}
is injective, therefore $|\cF|\leq|\cF_1\times\cF_2|=|\cF_1||\cF_2|$.
\end{proof}

We now have enough ingredients to prove \Cref{MainThmGeneralizedPEventown}.

\begin{proof}[Proof of \Cref{MainThmGeneralizedPEventown}]
The case $p=2$ is the Eventown Theorem. Suppose therefore $p\geq 3$. We use induction on $n$. 
Let $\cF$ be $p$-wise $p$-divisible. Since $\cF$ is assumed to be full-support
we have $n\geq p$. If $V^{\spn{p}}$ has a non-trivial stabiliser, then $V^{\spn{p}}=C_1\oplus C_2$ and we can
consider $\cF_1$ and $\cF_2$ given by \Cref{lem:C1+C2}. Let $n_1=|S_1|$ and
$n_2=|S_2|$. By the induction hypothesis $|\cF_1|\leq
2^{\lfloor n_1/p\rfloor}$ and
$|\cF_2|\leq 2^{\lfloor n_2/p\rfloor}$, and by \Cref{lem:C1+C2},
\[
|\cF|\leq|\cF_1||\cF_2|\leq
2^{\lfloor n_1/p\rfloor+\lfloor n_2/p\rfloor}\leq
2^{\lfloor (n_1+n_2)/p\rfloor}=2^{\lfloor n/p\rfloor}
\]
and we are done. It remains to consider the case when $\cF$ contains at least two non-zero subsets and $V^{\spn{p}}$ has a trivial
stabiliser, i.e. $\dim\operatorname{St}(V^{\spn{p}})=1$.

Applying \Cref{lem:kneserk} with $k=p$ we have
\[
\dim V^{\spn{p}} \geq p\dim V -p+1.
\]
Now suppose $|\cF|\geq 2^{\lfloor n/p\rfloor}+1$. Then, by \Cref{lem: improved Odlyzo}, $\dim V\geq \lfloor n/p\rfloor+2$. From this we get
\begin{align*}
\dim V^{\spn{p}}&\geq p\left(\left\lfloor\frac{n}{p}\right\rfloor+2\right) -p+1\\
&\geq p\left(\frac{n-(p-1)}{p}\right)+2p -p+1\\
&=n-(p-1)+p+1\\
&> n,
\end{align*}
a contradiction. Therefore, we must
have $|\cF|\leq 2^{\lfloor n/p\rfloor}.$
\end{proof}

Using a similar argument, we can prove \Cref{MainThmPEventownExtremal} regarding the extremal structure of $p$-divisible set families.

\begin{proof}[Proof of \Cref{MainThmPEventownExtremal}]
We proceed by induction on $n$. Let $\cF$ be $(p+1)$-wise $p$-divisible. 
We have $n\geq p$ since $\cF$ is assumed to be of full-support.
The case $\dim(V)=1$ is trivial. 

If $\dim V\geq 2$, we claim that $\dim\operatorname{St}(V^{\spn{p+1}})\geq 2$.
Assume the contrary, namely that $V^{\spn{p+1}}$ has trivial stabiliser: then \Cref{lem:p+1} implies that $\dim V\geq p+2$ and $\dim V^{\spn{p+1}}>(p+1)^2$. Hence, $n>(p+1)^2$. Let us first consider the case $p=2$. 
By \eqref{eq:Odlyzko} we have $2^{\lfloor n/2\rfloor-1}<|\cF|\leq |V\cap\{0,1\}^n|\leq 2^{\dim V}$. By \Cref{lem:kneserk} we therefore have
\[
\dim V^{\spn{3}}\geq 3\dim V-2\geq
3\left(\frac{n}{2}-\frac{1}{2}\right)-2=n+\frac 12(n-3)-2>n
\]
for $n>7$, a contradiction since we have $n>9$. This proves the claim for $p=2$. Suppose now $p\geq 3$. By \Cref{lem: improved Odlyzo}, we have
\[
2^{\lfloor n/p\rfloor -1}<|\cF|\leq |V\cap\{0,1\}^n|\leq 2^{\dim V-1}
\]
and thus $\dim V\geq\lfloor n/p\rfloor+1$. From this and \Cref{lem:kneserk} we get
\begin{align*}
\dim V^{\spn{p+1}}&\geq (p+1)\left(\left\lfloor\frac{n}{p}\right\rfloor+1\right) -p\\
&\geq (p+1)\left(\frac{n-(p-1)}{p}\right)+(p+1) -p\\
&\geq (p+1)\frac np -\frac{p^2-1}{p} +1\\
&> n + \frac np -p+1.
\end{align*}
Since $n>(p+1)^2$, we must
have $n/p > p$ and we obtain $\dim V^{\spn{p+1}}>n$, a contradiction. This proves the claim.

As $\dim\operatorname{St}(V^{\spn{p+1}})\geq 2$, there exist nonzero codes
$C_1,C_2$ such that $V^{\spn{p+1}}=C_1\oplus C_2$ by \Cref{StabiliserDecomposition}. Consider $\cF_1=\cF|_{S_1}$ and $\cF_2=\cF|_{S_2}$ where
$S_1=\supp(C_1)$ and $S_2=\supp(C_2)$. Let $n_1=|S_1|$ and $n_2=|S_2|$.
If $|\cF_1|\leq 2^{\lfloor n_1/p\rfloor-1}$ and $|\cF_2|\leq 2^{\lfloor
n_2/p\rfloor-1}$, then, applying the last statement of \Cref{lem:C1+C2},
\[
2^{\lfloor n/p\rfloor -1}<|\cF|\leq |\cF_1||\cF_2|\leq 2^{\lfloor n_1/p\rfloor+\lfloor n_2/p\rfloor-2},
\]
which means
\[
\left\lfloor \frac{n}{p}\right\rfloor<\left\lfloor \frac{n_1}{p}\right\rfloor+\left\lfloor \frac{n_2}{p}\right\rfloor -1,\]
a contradiction. Hence, without loss of generality, we may assume that $|\cF_1|> 2^{\lfloor n_1/p\rfloor-1}$. 
By \Cref{lem:C1+C2},  $\cF_1$ is $(p+1)$-wise $p$-divisible and we may
apply the induction hypothesis to $\cF_1$: we get that $\cF_1$ is a subfamily of an
atomic family consisting of some unions of $p$-element subsets. In particular
$n_1=ap$ for a positive integer $a$, and $|\cF_1|\leq 2^a$. Therefore,
\[
2^{\lfloor n/p\rfloor -1}<|\cF|\leq |\cF_1||\cF_2|\leq 2^a|\cF_2|
\]
and we therefore have $|\cF_2|> 2^{\lfloor n/p\rfloor -1 -a}=2^{\lfloor
(n-ap)/p\rfloor -1}=2^{\lfloor n_2/p\rfloor -1}$. Since $\cF_2$ is $(p+1)$-wise
$p$-divisible by \Cref{lem:C1+C2}, the induction hypothesis also applies to
$\cF_2$ and we have
that $\cF_2$ is also a subfamily of an atomic family with atoms of size $p$,
which proves the theorem. 
\end{proof}

It remains to prove the technical \Cref{lem: improved Odlyzo}.

\begin{proof}[Proof of \Cref{lem: improved Odlyzo}]
    As previously mentioned, after Gaussian elimination and a permutation of
coordinates, we obtain a matrix $G=[I_r | A]$ whose rows form a basis of $V$,
where $r=\operatorname{dim}V$ and $I_r$ is the $r\times r$ identity matrix. Let
$v_1,\ldots,v_r$ be the rows of $G$. 
 Note that by \Cref{lem:p+1}, we have $r\geq 4$ since $\operatorname{dim}\operatorname{St}(V^{\spn{3}})=1$. 

The vectors of $V$ are in one-to-one correspondence with linear combinations of the rows of $G$.
For $\blambda=(\lambda_1,\ldots,\lambda_r)\in\F_p^r$, let us write $v(\blambda)=\lambda_1v_1+\cdots
+\lambda_rv_r$.
Let us define
\[
\Lambda =\{\blambda\in\F_p^r:\; v(\blambda)\in V\cap\{0,1\}^{n}\}.
\]
We have already remarked that $v(\blambda)(j)\in\{0,1\}$ for $j=1\ldots r,$
implies that $\Lambda\subset\{0,1\}^r$. To further constrain $\Lambda$ we shall
focus on the remaining coordinates of $v(\blambda)$, namely $v(\blambda)(j), j>r$.

We shall be making repeated use of the following observation.

    \medskip

    \noindent\textbf{Observation}. Let $I\subset [r]$ and denote
$\overline{I}=[r]\setminus I$. Consider the subset of
$\blambda=(\lambda_1,\ldots,\lambda_r)\in\Lambda$ for
which the values of $\lambda_i$, $i\in\overline{I}$, are fixed to some quantity. Then, the possible
values of $(\lambda_{i})_{i\in I}$  satisfy
\[
\sum_{i\in I}\lambda_iv_i(j) \in\{k,k+1\}
\]
where $k=-\sum_{i\in\overline{I}}\lambda_iv_i(j)$. 

This is simply stating that
$v(\blambda)(j)=\sum_{i\in I}\lambda_iv_i(j) + \sum_{i\in\overline{I}}\lambda_iv_i(j)
\in\{0,1\}$.

Let us illustrate the usefulness of this observation with a simple example:
suppose column $j$ of the matrix $G$ has (at least) two non-zero elements
$v_s(j)$ and $v_t(j)$ in rows $s$ and $t$. 
Then, letting $(\lambda_s,\lambda_t)$ span $\{0,1\}^2$, we have that
$\lambda_sv_s(j)+\lambda_tv_t(j)$ must span two distinct non-zero values 
$(\bmod\; p)$, $p\geq 3$, which together with the $0$ element give at least three
values. The observation tells us therefore that for any fixed
$(\lambda_i)_{i\notin \{s,t\}}$, at most three values of $(\lambda_s,\lambda_t)$
are allowed that will yield $(\lambda_1,\ldots,\lambda_r)\in\Lambda$. We
therefore must have $|\Lambda|\leq\frac 342^r.$

To bring down $|\Lambda|$ to $\frac 122^r$, we will need to consider several
columns of $G$ simultaneously. We will work with a set of $3$ columns evaluated
on a common set $I$ of $4$ coordinates. We will use the hypothesis
$\dim\St(V^{\spn{3}})=1$ to ensure that the relevant submatrix of $G$ exists.
We divide the proof into four steps. The first step ensures that we may
suppose all entries of $G$ to be in $\{0,1,-1\}$. The second step tells us that we
may assume that no $2\times 2$ submatrix of $G$ has only non-zero entries with
one distinct from the three others. The third step exhibits the existence of a
$4\times 3$ submatrix of $G$ with the required properties. The fourth and final
step applies the observation to the columns of this submatrix to prove that
$|\Lambda|\leq 2^{r-1}$.

    \medskip

    \noindent\textbf{Claim 1}. {\em All entries of $G$ are in $\{-1,0,1\}$ or
else
$|\Lambda|\leq 2^{r-1}$.}

For every fixed choice of
$(\lambda_1,\ldots,\lambda_{i-1},\lambda_i,\ldots,\lambda_r)$, by the
\textbf{Observation} above, we have $\lambda_iv_i(j)\in\{k,k+1\}$ for some integer
$k \bmod p$. 
 On the other hand, since
$\lambda_i\in\{0,1\}$, we have $\lambda_iv_i(j)\in\{0,v_i(j)\}$. Thus, if $v_i(j)\notin\{-1,0,1\}$, 
then the two values $\lambda_i=0$ and $\lambda_i=1$ cannot evaluate at two
consecutive integers $k,k+1$ mod $p$, for any $k$. Therefore, at most one value of
$\lambda_i$ is allowable for every choice of
$(\lambda_1,\ldots,\lambda_{i-1},\lambda_{i+1},\ldots,\lambda_r)$, and we have
$|\Lambda|\leq \frac 12 2^r=2^{r-1}$.

\medskip

    From now on, we therefore assume that $v_i(j)\in\{-1,0,1\}$ for all $i,j$.

    \medskip

    \noindent\textbf{Claim 2}. {\em If there exist two distinct $i,j$ such that
$v_i(s)\neq v_j(s)$ and $v_i(t)=v_j(t)$ for two distinct
$s,t\in\supp(v_iv_j)$, then $|\Lambda|\leq 2^{r-1}$.}

    Indeed, by fixing all $\lambda_h$ for $h\neq i,j$ and considering the $s$-th and $r$-th columns of $G$, the \textbf{Observation} gives us
    \begin{align*}
    \left\{
    \begin{array}{cc}
        \lambda_i-\lambda_j\in\{k_1,k_1+1\}, \\
        \lambda_i+\lambda_j\in\{k_2,k_2+1\},
    \end{array}
    \right.
    \end{align*}
    for $k_1,k_2\in\Z/p\Z$. Since $\lambda_i,\lambda_j\in\{0,1\}$, if
$\lambda_i+\lambda_j$ can only be equal to two consecutive integers mod $p$, then $(\lambda_i,\lambda_j)$ 
can only take the three values
\[
(0,0),(1,0),(0,1)\quad\text{or}\quad (1,0),(0,1),(1,1).
\]
 These three choices give three consecutive values of $\lambda_i-\lambda_j$.
Hence, the pair $(\lambda_i,\lambda_j)$ can take at most two values, which
implies $|\Lambda|\leq \frac 242^r=2^{r-1}.$

\medskip

    Let $\one_{\operatorname{supp}(v_iv_j)}$ be the characteristic vector of
$\supp(v_iv_j)$. By Claim 2, from now on we may assume that
    \begin{equation}\label{eq:vivj}
    v_iv_j=\pm\one_{\operatorname{supp}(v_iv_j)},\quad\forall i\neq j.
    \end{equation}

    \medskip

    \noindent\textbf{Claim 3}. {\em There exist $1\leq i_1,i_2,i_3,i_4\leq r$ such that
    \begin{itemize}
    \item $v_{i_1}*v_{i_2}*v_{i_3}*v_{i_4}=v_{i_1} v_{i_2} v_{i_3} v_{i_4}\neq 0$,
    \item $\supp(v_{i_1} v_{i_2} v_{i_3} v_{i_4})\subsetneq\supp(v_{i_1} v_{i_2}
v_{i_3})\subsetneq\supp(v_{i_1} v_{i_2})\subsetneq\supp(v_{i_1})$.
    \end{itemize}}

    \begin{figure}[H]
    \begin{center}
    \begin{tikzpicture}[scale=0.8]
    \draw[thick] (0,4) -- (4,4); 
    \draw[thick] (0,3) -- (3,3); \draw[thick, dashed, color=blue] (4,3) -- (6,3);
    \draw[thick] (0,2) -- (2,2); \draw[thick, dashed, color=blue] (3,2) -- (6,2);
    \draw[thick] (0,1) -- (1,1); \draw[thick, dashed, color=blue] (2,1) -- (6,1);
    \node at (-0.5,4) {$v_{i_1}$};
    \node at (-0.5,3) {$v_{i_2}$};
    \node at (-0.5,2) {$v_{i_3}$};
    \node at (-0.5,1) {$v_{i_4}$};
    \end{tikzpicture}
    \end{center}
    \caption{The four vectors of claim 3}
    \label{Four vectors in Step 5}
    \end{figure}
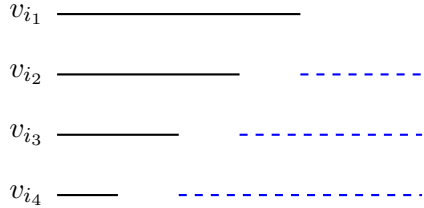

    \begin{proof}[Proof of the claim]
        We will determine these four vectors sequentially.

        \medskip

Recall that the hypothesis $\dim\St(V^{\spn{3}})=1$ means that the only non-zero
stabiliser of $V^{\spn{3}}$ is the all-one vector.

\noindent\textbf{Step 3.1}. Let $v_{i_1}$ be any row of $G$ and let
$S=\supp(v_{i_1})$. Suppose that there
does not exist $i_2$ such that
$\varnothing\subsetneq\supp(v_{i_1}v_{i_2})\subsetneq S$. Then, for every
$i$, either $v_iv_{i_1}=0$, or $\supp(v_iv_{i_1})=S$, in which case
\eqref{eq:vivj} implies $v_iv_{i_1}=\pm\one_S$, or equivalently, $v_{i}\one_S=\pm
v_{i_1}$. But this means that $\one_S$ is a non-trivial stabiliser of $V$, and
therefore also of $V^{\spn{3}}$, a contradiction.

        \medskip

\noindent\textbf{Step 3.2}. Let $i_1,i_2$ be such that
$\varnothing\subsetneq\supp(v_{i_1}v_{i_2})\subsetneq \supp(v_{i_1})$ and let
$S=\supp(v_{i_1}v_{i_2}).$ Suppose there does not exist $i_3$ such that 
 $\varnothing\subsetneq\supp(v_{i_1}v_{i_2}v_{i_3})\subsetneq S$. Then, for every
$i$ we have either $v_i\one_S=0$, or $\supp(v_i)\supset S$. Therefore, for any
$i,j$, we have either $v_{i}v_{j}\one_S=0$, or $\supp(v_iv_j)\supset S$. But in this
last case, \eqref{eq:vivj} implies that $v_iv_j\one_S=\pm\one_S=\pm
v_{i_1}v_{i_2}$. Therefore, for every $i,j$ we have $v_iv_j\one_S\in
V^{\spn{2}}$, which means that $\one_S$ stabilises $V^{\spn{2}}$ and hence also
$V^{\spn{3}}$, a contradiction.

        \medskip

\noindent\textbf{Step 3.3}. Let $i_1,i_2,i_3$ be such that
$\varnothing\subsetneq\supp(v_{i_1}v_{i_2}v_{i_3})\subsetneq
\supp(v_{i_1}v_{i_2})\subsetneq\supp(v_{i_1})$ and let
$S=\supp(v_{i_1}v_{i_2}v_{i_3}).$ Suppose there does not exist $i_4$ such that 
$\varnothing\subsetneq\supp(v_{i_1}v_{i_2}v_{i_3}v_{i_4})\subsetneq S$.
Then, for every
$i$ we have either $v_i\one_S=0$, or $\supp(v_i)\supset S$.  Therefore, for any
$i,j,k$, we have either $v_iv_jv_k\one_S=0$, or $\supp(v_iv_jv_k)\supset S$. But
in this last case \eqref{eq:vivj} implies that
$v_iv_jv_k\one_S=\pm v_{i_1}v_{i_2}v_{i_3}$ which means that $\one_S$ is a
non-trivial stabiliser of $V^{\spn{3}}$, a contradiction.
\end{proof}

    \noindent\textbf{Final Step}. For ease of notation, let
us write the indices of the four vectors of Claim 3 as
$i_1=1,i_2=2,i_3=3$ and $i_4=4$. For each choice of
$(\lambda_5,\lambda_6,\ldots,\lambda_r)$, claims 1,2,3 imply that there exist $\delta_2,\delta_3,\delta_4\in\{1,-1\}$ such that either
    \begin{align}\label{eq1:lem:improved Odlyzo}
        \left\{
        \begin{array}{cc}
        \lambda_1+\delta_2\lambda_2+\delta_3\lambda_3+\delta_4\lambda_4\in\{k_1,k_1+1\}\\
        \lambda_1+\delta_2\lambda_2+\delta_3\lambda_3\in\{k_2,k_2+1\}\\
        \lambda_1+\delta_2\lambda_2\in\{k_3,k_3+1\}
        \end{array}
        \right.
    \end{align}
    for $k_1,k_2,k_3\in\Z/p\Z$, or
    \begin{align}\label{eq2:lem:improved Odlyzo}
        \left\{
        \begin{array}{cc}
        \lambda_1+\delta_2\lambda_2+\delta_3\lambda_3+\delta_4\lambda_4\in\{k_1',k_1'+1\}\\
        \lambda_1+\delta_2\lambda_2+\delta_3\lambda_3\in\{k_2',k_2'+1\}\\
        \lambda_1+\delta_2\lambda_2+\delta_4\lambda_4\in\{k_3',k_3'+1\}
        \end{array}
        \right.
    \end{align}
    for $k_1',k_2',k_3'\in\Z/p\Z$. We aim to determine the maximal number of
solutions $(\lambda_1,\lambda_2,\lambda_3,\lambda_4)$ for these systems. We
begin by examining the first constraint in the first case, which states that: 
    \[
    \lambda_1+\delta_2\lambda_2+\delta_3\lambda_3\in\{k_1,k_1+1\}
    \]
    when $\lambda_4=0$ and
    \[
    \lambda_1+\delta_2\lambda_2+\delta_3\lambda_3\in\{(k_1-\delta_4),(k_1-\delta_4)+1\}
    \]
 when
$\lambda_4=1$. In other words, $\lambda_1+\delta_2\lambda_2+\delta_3\lambda_3$
must belong to {\em different} pairs of consecutive integers $\bmod\, p$ for
$\lambda_4=0$ and for $\lambda_4=1$. But
the second constraint states that
$\lambda_1+\delta_2\lambda_3+\delta_3\lambda_3$ can only belong to a {\em
constant} pair $(k_2,k_2+1)$ of
consecutive integers $\bmod\, p$. We therefore have that the maximum number of solutions
for $(\lambda_1,\lambda_2,\lambda_3,\lambda_4)$ is obtained when
    \begin{align}\label{eq3:lem:improved Odlyzo}
        \left\{
        \begin{array}{cc}
        \lambda_1+\delta_2\lambda_2+\delta_3\lambda_3\in\{k_2,k_2+1\}\\
        \lambda_1+\delta_2\lambda_2\in\{k_3,k_3+1\}
        \end{array}
        \right.
    \end{align}
    for one choice of $\lambda_4$, and
        \begin{align}\label{eq4:lem:improved Odlyzo}
        \left\{
        \begin{array}{cc}
        \lambda_1+\delta_2\lambda_2+\delta_3\lambda_3\quad\text{ is a fixed
integer $\bmod\, p$}\\
        \lambda_1+\delta_2\lambda_2\in\{k_3,k_3+1\}
        \end{array}
        \right.
    \end{align}
    for the other choice of $\lambda_4$. The system (\ref{eq4:lem:improved
Odlyzo}) has at most $3$ solutions for $(\lambda_1,\lambda_2,\lambda_3)$, as
$\lambda_1+\delta_2\lambda_2\in\{k_3,k_3+1\}$ has at most $3$ solutions for
$(\lambda_1,\lambda_2)$ and there is at most one allowable value of $\lambda_3$
for every $(\lambda_1,\lambda_2)$. Let us now derive an upper bound on
the number of solutions for (\ref{eq3:lem:improved Odlyzo}). Using
essentially the same argument as above, for one choice of $\lambda_3$,
    \[
    \lambda_1+\delta_2\lambda_2\quad\text{is a fixed integer $\bmod\, p$}
    \]
    and so $(\lambda_1,\lambda_2)$ has at most $2$ solutions; and for the other choice of $\lambda_3$,
    \[
    \lambda_1+\delta_2\lambda_2\in\{k_3,k_3+1\}
    \]
    and so $(\lambda_1,\lambda_2)$ has at most $3$ solutions, giving at most $5$
solutions for \eqref{eq3:lem:improved Odlyzo}. 
In total, the system (\ref{eq1:lem:improved Odlyzo}) has therefore at most $3+5=8$
solutions for $(\lambda_1,\lambda_2,\lambda_3,\lambda_4)$. An essentially
identical argument gives the same result for the system (\ref{eq2:lem:improved Odlyzo}). Thus, 
$|\Lambda|\leq 8\cdot 2^{r-4}=2^{r-1}$ which proves the lemma.
\end{proof}

\section{$\ell$-divisible set families for composite $\ell$}
\label{section:Establishing the atomicity of k-closed families}

This section is devoted to proving \Cref{MainThmGeneral}, the structure theorem for $\ell$-divisible set families.

When $\ell$ is a arbitrary integer, we encounter two difficulties compared to the
prime case. One is that when considering the vector space $V$ generated by $\cF$
over $\F_p$, for a prime $p$, then \Cref{codes k closed} cannot tell us whether
the vectors of $V\cap\{0,1\}^n$ have a support of cardinality divisible by a
power $p^\alpha$ of $p$. Switching to the vector space over $\F_{p^\alpha}$ does
not help. To deal with this problem, we borrow an idea from
\cite{Gishboliner2022Small} which is captured by \Cref{lem:prime power} below.
The second issue is that when we define the vector space $V$ generated by $\cF$
over $\F_p$, for $p$ a prime divisor of $\ell$, and
 we try to argue that $V^{\spn{k}}$ breaks into a direct
sum of spaces, we may get different decompositions for different prime
divisors of $\ell$ and reconciling them may not be obvious.

To address this second difficulty, we adopt a strategy also present in
\cite{Gishboliner2022Small}. We will study the {\em atoms} of $\cF$ (called
maximal sets of twins in \cite{Gishboliner2022Small}), namely the
maximal subsets of $[n]$ on which the functions of $\cF$ are constant.
The strategy is to prove the existence of an atom of cardinality divisible by
$\ell$ and apply an induction argument.

We first make some statements about the atoms of a family $\cF$.
Recall that we regularly identify $\cF\in 2^{[n]}$ with $\cF\subset \{0,1\}^n$ and every $A\in\cF$ with its characteristic vector $\one_A\in\{0,1\}^n$.

\begin{defn}[atom]
    Let $\cF\subset 2^{[n]}$. An \textit{atom} of $\cF$ is a non-empty subset $A$ of $\supp(\cF)\subset[n]$ satisfying
\begin{enumerate}
\item[(i)] $\forall F\in\cF,\quad
\text{either }A\subset F\text{ or }F\cap A=\varnothing$,
\item[(ii)]  $\forall B\supsetneq A$, $\exists F\in\cF, \quad \varnothing\subsetneq F\cap
B\subsetneq B$.
\end{enumerate}
In words, an atom is a set $A$ satisfying (i) and maximal for inclusion with this property. 
\end{defn}

Clearly, the atoms of $\cF$ form a partition of $\supp(\cF)$ and the family $\cF$ is
included in the atomic family whose atoms are the atoms of $\cF$. In particular
$|\cF|\leq 2^a$ if $a$ is the number of atoms of $\cF$.

The following proposition is straightforward.

\begin{prop}\label{prop:Fr F same atoms}
    Let $\cF\subset 2^{[n]}$. Then for every $r\geq 1$, the set family $\cF^r$ has the same atoms as $\cF$.
\end{prop}

\begin{prop}\label{prop:cFa contains an atom of F}
    Let $\cF\subset 2^{[n]}$ and let $a$ be the number of atoms of $\cF$. Then $\cF^a$ contains an atom of $\cF$.
\end{prop}
\begin{proof}
    Let $\F$ be a field and let $V$ be the $\F$-vector space generated by $\cF$. Let $k\coloneqq\dim V$. Then $k\leq a$ and there exist $F_1,\ldots,F_k\in\cF$ such that $\one_{F_1},\ldots,\one_{F_k}$ is a basis of $V$. Consider an inclusion-minimal non-empty subset in $\{F_{i_1}\cap\cdots\cap F_{i_j}:1\leq i_1<\ldots<i_j\leq k\}$, which we denote by $A$. Then
    \[
    \text{either }A\subset F_i\text{ or }F_i\cap A=\varnothing,\quad\forall 1\leq i\leq k
    \]
    and $A$ is inclusion-maximal with this property. Since
$\one_{F_1},\ldots,\one_{F_k}$ is a basis of $V$, we have that any vector in $V$
is constant on $A$, therefore $A$ is an atom of $\cF$.
\end{proof}

\begin{lem}\label{lem:dim Ci=1}
    Let $r\in\Z_{>0}$ and $\F$ be a field. Let $\cF\subset 2^{[n]}$ and $V$ be the $\F$-vector space generated by $\cF$. Suppose
    \[
    V^{\spn{r}}=C_1\oplus\cdots\oplus C_m
    \]
    for an integer $m$ and some nonzero codes $C_i$. If $\dim C_i=1$, then $\supp(C_i)$ is an atom of $\cF$.
\end{lem}
\begin{proof}
    Note that $V^{\spn{r}}$ is generated by $\cF^r$. Since $\dim C_i=1$, we have that $\supp(C_i)$ is an atom of $\cF^r$ and thus the result follows from \Cref{prop:Fr F same atoms}.
\end{proof}

The following two lemmas will be useful for induction arguments. The next lemma
is straightforward.

\begin{lem}\label{lem:k wise ell divisible outside atom}
    Let $k,\ell\in\Z_{>0}$ and let $\cF\subset 2^{[n]}$ be $k$-wise $\ell$-divisible. Let $A$ be an atom of $\cF$ having cardinality divisible by $\ell$. Then $\cF|_{[n]\backslash A}$ is $k$-wise $\ell$-divisible.
\end{lem}

\begin{lem}\label{lem:MainThmGeneral large k}
    Let $k,\ell\in\Z_{>0}$. Let $\cF\subset 2^{[n]}$ be a full-support $k$-wise
$\ell$-divisible set family. Suppose $k\geq n$. Then $\cF$ has an atom of
size divisible by $\ell$.
\end{lem}
\begin{proof}
Let $a$ be the number of atoms of $\cF$. We have $k\geq n\geq a$ and therefore
$\cF^k\supset\cF^a$ contains an atom $A$ of $\cF$ by \Cref{prop:cFa contains an
atom of F}. Elements of $\cF^k$ have cardinality divisible by $\ell$, by
definition of $k$-wise $\ell$-divisibility, therefore $|A|$ is divisible by
$\ell$.
\end{proof}

The next lemma enables us to deal with the case when $\ell$ is a prime power.

\begin{lem}\label{lem:prime power}
    Let $\alpha,k\in\Z_{>0}$ and $p$ be a prime integer. Let $\cF\subset 2^{[n]}$ be $k\phi(p^\alpha)$-wise $p^\alpha$-divisible, where $\phi$ is the Euler's totient function, 
    and let $V$ be the $\F_p$-vector space generated by $\cF$. Let $v\in V^{\spn{k}}\cap\{0,1\}^n$ and let $S=\supp(v)$. Then $|S|$ is divisible by $p^\alpha$.
\end{lem}
\begin{proof}
    Since $v\in V^{\spn{k}}$, there exist $\lambda_1,\ldots,\lambda_f\in\Z$ and $v_1,\ldots,v_f\in\cF^k$ such that
    \[
    v\equiv \lambda_1v_1+\cdots+\lambda_fv_f\ (\text{mod }p).
    \]
    We define $w\in\Z^n$ such that
    \begin{equation}\label{eq:w}
    w = \lambda_1v_1+\lambda_2v_2+\cdots+ \lambda_fv_f.
    \end{equation}
    From \eqref{eq:w} we have that $w^{\phi(p^\alpha)}$ is a linear combination of some elements of $\cF^{k\phi(p^\alpha)}$ over $\Z$. As $\cF$ is $k\phi(p^\alpha)$-wise $p^\alpha$-divisible, all elements of $\cF^{k\phi(p^\alpha)}$ have the sum of their coordinates divisible by $p^\alpha$, and so do their linear combinations. Therefore,
    \begin{equation}\label{eq:palpha}
    \sum_{i=1}^nw(i)^{\phi(p^\alpha)} \equiv0\ ( \bmod\ p^\alpha).
    \end{equation}
    Since $v\in\{0,1\}^n$, we have that $w(i)\equiv 1\ (\text{mod }p)$ if $i\in S$ and $w(i)\equiv 0\ (\text{mod }p)$ if $i\notin S$: and since $\phi(p^\alpha)\geq\alpha$, we have $p^\alpha\ |\ p^{\phi(p^\alpha)} |\ w(i)^{\phi(p^\alpha)}$ for $i\notin S$. 
    By Fermat-Euler's theorem, we have $w(i)^{\phi(p^\alpha)}\equiv 1 \ (\text{mod }p^\alpha)$ for $i\in S$. 
    Summing, we therefore obtain
    \[
    \sum_{i=1}^nw(i)^{\phi(p^\alpha)} \equiv |S|\ ( \bmod\ p^\alpha)
    \]
    which together with \eqref{eq:palpha} gives us that $|S|\equiv 0\ (  \bmod\ p^\alpha)$.
\end{proof}

Let $\ell=p_1^{\alpha_1}\cdots p_h^{\alpha_h}$ be the prime factorization of the
integer
$\ell$, and let $\cF$ be a $k$-wise $\ell$-divisible family for a sufficiently
large $k$.
 The next two lemmas will allow us to introduce subsets $S_i$ of $[n]$ whose complement is
a union of atoms of $\cF$ of size divisible by $p_i^{\alpha_i}$. The core of the proof of
\Cref{MainThmGeneral} will consist of showing that the union of the $S_i$ is
not the whole set $[n]$, so that there must exist an atom of $\cF$ of
cardinality divisible by $p_i^{\alpha_i}$ for every $i=1\ldots h$.

\begin{lem}\label{lem:1/2}
    Let $t\in \Z_{>0}$. Let $V$ be a subspace of $\F_p^n$. Let $m=\dim\operatorname{St}(V^{\spn{t}})$, and let
\[
V^{\spn{t}}=C_1\oplus C_2\oplus\cdots\oplus C_m
\]
be the corresponding decomposition of $V^{\spn{t}}$ given in \Cref{StabiliserDecomposition}. 
Let $I=\{i:1\leq i\leq m,\ \dim C_i\geq 2\}$ and let $S=\bigcup_{i\in I}\supp(C_i)$.
Let $W=V|_S$.
Then,
\[
\dim V^{\spn{r}}\geq \dim V^{\spn{r-1}} +\frac{1}{2} \dim W
\]
for all $r$ with $2\leq r\leq t$.
\end{lem}

\begin{proof}
Let us first prove the result for $r=t$. For $i=1\ldots m$, let $V_i\coloneqq V|_{\supp(C_i)}$ be
the restriction of the vector space $V$ to the coordinates of $\supp(C_i)$. When
necessary, we allow
ourselves to identify $V_i$ with a subspace of $\F_p^n$ by padding the vectors
of $V_i$ with zeros outside $\supp(C_i)$. Notice that $V_i^{\spn{t}}=C_i$.
For $i\in I$ we have $\dim V_i\geq 2$, otherwise we would have $\dim V_i^{\spn{t}}=\dim C_i=1$, which 
contradicts the definition of $I$. By definition of the $C_i$'s we have that  $V_i^{\spn{t}}$ has
trivial stabiliser and by Kneser's Theorem \ref{Kneser} we have 
\begin{align*}
\dim V_i^{\spn{t}}&\geq \dim V_i^{\spn{t-1}} +\dim V_i -1\\
&\geq \dim V_i^{\spn{t-1}} +\frac{1}{2}\dim V_i
\end{align*}
for $i\in I$ and $\dim V_i^{\spn{t}}\geq \dim V_i^{\spn{t-1}}$ for $i\not\in
I$. Therefore,
\begin{align*}
\dim V^{\spn{t}} &= \sum_{i=1}^m \dim V_i^{\spn{t}}\\
                   &\geq \sum_{i\not\in I}\dim  V_i^{\spn{t-1}} +\sum_{i\in I}\dim
V_i^{\spn{t-1}} + \frac{1}{2}\sum_{i\in I}\dim V_i\\
&= \sum_{i=1}^m\dim V_i^{\spn{t-1}}+\frac{1}{2}\sum_{i\in I}\dim V_i\\
&\geq \dim V^{\spn{t-1}} + \frac{1}{2}\dim W
\end{align*}
with the last inequality coming from the fact that $V^{\spn{t-1}}\subset V_1^{\spn{t-1}}
\oplus \cdots \oplus  V_m^{\spn{t-1}}$ and $W\subset \bigoplus_{i\in I}V_i$.

For $r<t$ we have therefore 
\[
\dim V^{\spn{r}}\geq \dim V^{\spn{r-1}} +\frac{1}{2} \dim W_r
\]
where $W_r$ is defined according the decomposition of $V^{\spn{r}}$ induced by
the stabiliser of $V^{\spn{r}}$. However, we have that
$\operatorname{St}(V^{\spn{j}})\subset\operatorname{St}(V^{\spn{j+1}})$ for any $j$, so that
$\operatorname{St}(V^{\spn{r}})\subset\operatorname{St}(V^{\spn{t}})$. This implies that $W_r\supset W$,
hence the result.
\end{proof}

\begin{lem}\label{lem:tphi}
    Let $t,\alpha\in \Z_{>0}$ and let $p$ be a prime number. Let
$\cF\subset\{0,1\}^n$ be a full-support $k$-wise $\ell$-divisible family
for integers $k,\ell$ such that $p^\alpha | \ell$ and $k\geq t\phi(p^\alpha)$. Let $V$ be the $\F_p$-vector space generated by $\cF$ 
and let 
\[
V^{\spn{t}}=C_1\oplus C_2\oplus\cdots\oplus C_m
\]
be the StabiliserDecomposition of $V^{\spn{t}}$ given by \Cref{StabiliserDecomposition} 
for $m=\dim\operatorname{St}(V^{\spn{t}})$. 
Let $I=\{i:1\leq i\leq m,\ \dim C_i\geq 2\}$ and let $S=\bigcup_{i\in I}\supp(C_i)$.
Then, every $j\in [n]\setminus S$ belongs to an atom of $\cF$ of cardinality divisible by
$p^\alpha$.
\end{lem}

\begin{proof}
By definition of $S$, every $j\in [n]\setminus S$ belongs to some $A_i=\supp(C_i)$ with
$\dim C_i=1$. By \Cref{lem:dim Ci=1}, we have that $A_i$ is an atom of $\cF$.
Furthermore, $\one_{A_i}\in C_i$ and by \Cref{codes k closed} we have that
$|A_i|$ is divisible by $p$. Finally, by \Cref{lem:prime power}
we have that $|A_i|$ is divisible by $p^\alpha$.
\end{proof}

\begin{proof}[Proof of \Cref{MainThmGeneral}]
    Using essentially the same argument as for \Cref{rmk:assume full support} for the prime case, we may assume that $\cF\subset\{0,1\}^n$ is a full-support set family. 
We shall prove the statement:

{\em If $\cF$ is $4\ell^2$-wise $\ell$-divisible and if $|\cF|> 2^{\lfloor
n/\ell\rfloor-1}$, then $\cF$ contains an atom $A$ of cardinality $|A|$ divisible by $\ell$.}

The theorem follows from the statement by induction. Indeed, if $A$ is such an
atom, then $\cF'\coloneqq\cF|_{[n]\setminus A}$ is $4\ell^2$-wise
$\ell$-divisible by \Cref{lem:k wise ell divisible outside atom}. Furthermore,
we have $|\cF|\leq 2|\cF'|$ so that $|\cF'|>2^{\lfloor n'/\ell\rfloor-1}$ where
$n'=n-|A|$. Therefore $\cF'$ also satisfies the hypothesis of the statement and
we obtain inductively that all the atoms of $\cF$ have cardinality divisible by
$\ell$. It follows that $n$ is a multiple of $\ell$ and $|\cF|\leq 2^{\lfloor n/\ell\rfloor}$ since for
any family $\cF$ we have $|\cF|\leq 2^a$ whenever $a$ is the number of its
atoms. This last argument also shows that if $|\cF|> 2^{\lfloor n/\ell\rfloor-1}$,
then all atoms of $\cF$ have size exactly $\ell$.

It remains to prove the statement.
Consider the prime factorization $\ell=p_1^{\alpha_1}\cdots p_h^{\alpha_h}$,
where $p_1,\ldots,p_h$ are pairwise distinct. Let $\cF$ be $k$-wise
$\ell$-divisible, where $k=4\ell^2$. Set $t=4\ell h$, so that $k=t\ell/h$. Let $\spn{\cF}_{p_i}$ be the vector space generated by $\cF$ over $\F_{p_i}$ for every $i$.
    
    If $n\leq 2\ell$, then $k=4\ell^2\geq n$ and thus the conclusion follows from \Cref{lem:MainThmGeneral large k}. Now suppose $n>2\ell$. Then
    \begin{equation}\label{eq:n>4l}
        \left\lfloor \frac{n}{2\ell} \right\rfloor \leq \left\lfloor \frac{n}{\ell} \right\rfloor -1.
    \end{equation}
  Note that 
\begin{equation}
\label{eq:dimV>n/t}
\dim\spn{\cF}_{p_i}>\frac nt \quad \text{for all $i$,}
\end{equation}
otherwise there exists $i$ such that $\dim\spn{\cF}_{p_i}\leq n/t$ and then (\ref{eq:Odlyzko}) implies that
    \[
    2^{\lfloor n/\ell\rfloor-1}<|\mathcal{F}|\leq|\spn{F}_{p_i}\cap\{0,1\}^n|\leq 2^{\lfloor n/4\ell h\rfloor},
    \]
    which contradicts (\ref{eq:n>4l}). Let $S_i$ be defined as the set $S$ in
\Cref{lem:tphi} for the prime $p_i$, namely $S_i$ is
 the union of the supports of direct summands with dimension greater than $2$ in
the StabiliserDecomposition of $\spn{\cF}_{p_i}^{\spn{t}}$. 
Let us write $\cF_i=\cF|_{S_i}$ for $i=1\ldots h$. Let  $W_i=\spn{\cF_{i}}_{p_i}$. 
Then
    \begin{equation}\label{eq:dimW}
        \dim W_i<\frac{2n}{t}=\frac{n}{2\ell h},
    \end{equation}
    otherwise, writing $V=\langle\cF\rangle_{p_i}$, we have $\dim V^{\spn{r}}-\dim V^{\spn{r-1}}\geq\frac{n}{t}$ for all
$2\leq r\leq t$ by \Cref{lem:1/2}, which together with \eqref{eq:dimV>n/t}
contradicts $\dim V^{\spn{t}}\leq n$.

Let $U=\bigcup_{i=1}^hS_i$. The map
\begin{eqnarray*}
\cF|_{U} & \rightarrow & \cF_1\times\cF_2\times \cdots\times\cF_h\\
    F    & \mapsto     & (F\cap S_1,F\cap S_2,\ldots ,F\cap S_h)
\end{eqnarray*}
is injective, therefore, by (\ref{eq:Odlyzko}), (\ref{eq:n>4l}) and (\ref{eq:dimW}) we have that
    \[
    |\cF|_{U}|\leq \prod_{i=1}^h|\cF_i|\leq 2^{\dim W_1+\cdots+\dim W_h}\leq 2^{\left\lfloor n/2\ell\right\rfloor}\leq 2^{\lfloor n/\ell\rfloor -1}.
    \]
    But we have supposed $|\cF|>2^{\lfloor n/\ell\rfloor-1}$. Therefore, there exists $j\in [n]\backslash U$. 
It is readily checked that the inequality $\ell\geq h\phi(p_i^{\alpha_i})$
always holds, so that we have
 \[
    k=\frac{t\ell}{h}\geq t\phi(p_i^{\alpha_i})\quad\text{for every $i=1\ldots
h$.}
 \]
Applying \Cref{lem:tphi}, we therefore have that the atom $A$ of $\cF$
containing $j$ has cardinality divisible by $p_i^{\alpha_i}$ for every $i=1\ldots h$,
meaning that $|A|$ is divisible by 
$\ell$, which concludes the proof.
\end{proof}

\section{Concluding comments}
\label{sec:conclusion}

Although \Cref{MainThmGeneralizedPEventown} is optimal for $p=3$, in the sense that it supposes families $\cF$ to be $k$-wise $p$-divisible for the smallest possible value of $k$, we do not know whether \Cref{MainThmPEventownExtremal} is likewise optimal, even for $p=3$. In other words, we do not know whether there exist non-atomic $3$-wise $3$-divisible set families $\cF\subset 2^{[n]}$ that achieve the bound $2^{\lfloor n/3\rfloor}$.

We have not tried to optimise the constant in the term $4\ell^2$ in \Cref{MainThmGeneral} because it seems unlikely to us that it captures the right order of magnitude. The $O(\ell^2)$ behaviour cannot be brought down solely with the methods of this paper however, since using \Cref{lem:prime power} forces us to consider $k$-wise $\ell$-divisible families with $k=\Omega(\ell^2)$ when $\ell$ is a power of a prime. The term $k=4\ell^2$ captures therefore the worst-case behaviour of this paper's methods, but can be improved when $\ell$ has special forms. In particular, if $\ell$ is a square-free integer, then \Cref{lem:prime power} is not needed and \Cref{MainThmGeneral} becomes valid for $k$-wise $\ell$-divisible families with $k=O(\ell\omega(\ell))$, where $\omega(\ell)$ denotes the number of prime factors of $\ell$, known to behave as $\log\ell/\log\log\ell$ \cite{RobinAnalytic}.

\bigskip

\noindent
\textbf{Acknowledgment:}
The authors are grateful to Oriol Serra for bringing to their attention \cite{Gishboliner2022Small} and suggesting the line of work pursued in this paper. Chenying Lin was supported by the SFB 1085 funded by DFG and the MICINN research project PID2023-147642NB-I00.

\newpage
\bibliographystyle{alpha}
\bibliography{ref.bib}

\begin{thebibliography}{GST22}

\bibitem[Ber69]{Eventown2}
E.~R. Berlekamp.
\newblock On subsets with intersections of even cardinality.
\newblock {\em Can. Math. Bull.}, 12:471--474, 1969.

\bibitem[BF22]{babai}
L{\'a}szl{\'o} Babai and P{\'e}ter Frankl.
\newblock {\em Linear algebra methods in combinatorics}.
\newblock 2022.

\bibitem[BL17]{beck2017}
Vincent Beck and C{\'e}dric Lecouvey.
\newblock Additive combinatorics methods in associative algebras.
\newblock {\em Confluentes Mathematici}, 9(1):3--27, 2017.

\bibitem[FO83]{FranklOdlyzko}
P.~Frankl and A.~M. Odlyzko.
\newblock On subsets with cardinalities of intersections divisible by a fixed integer.
\newblock {\em Eur. J. Comb.}, 4:215--220, 1983.

\bibitem[Gra75]{Eventown1}
J.~E. Graver.
\newblock Boolean designs and self-dual matroids.
\newblock {\em Linear Algebra Appl.}, 10:111--128, 1975.

\bibitem[GST22]{Gishboliner2022Small}
L.~Gishboliner, B.~Sudakov, and I.~Tomon.
\newblock Small doubling, atomic structure and {{\(\ell\)}}-divisible set families.
\newblock {\em Discrete Analysis}, sep 30 2022.

\bibitem[Kne53]{kneser1953}
Martin Kneser.
\newblock Absch{\"a}tzung der asymptotischen dichte von summenmengen.
\newblock {\em Mathematische Zeitschrift}, 58(1):459--484, 1953.

\bibitem[MZ15]{KneserCode}
Diego Mirandola and Gilles Z{\'e}mor.
\newblock Critical pairs for the product {Singleton} bound.
\newblock {\em IEEE Trans. Inf. Theory}, 61(9):4928--4937, 2015.

\bibitem[Odl81]{dimV}
A.~M. Odlyzko.
\newblock On the ranks of some (0,1)-matrices with constant row sums.
\newblock {\em J. Aust. Math. Soc., Ser. A}, 31:193--201, 1981.

\bibitem[Ran15]{randriam15}
Hugues Randriambololona.
\newblock On products and powers of linear codes under componentwise multiplication.
\newblock {\em Algorithmic arithmetic, geometry, and coding theory}, 637(3-78):32, 2015.

\bibitem[Rob83]{RobinAnalytic}
Guy Robin.
\newblock Estimate of the {Chebyshev} theta function on the {{\(k\)}}th prime number and large values of the number of prime divisors function {{\(\omega(n)\)}} of {{\(n\)}}.
\newblock {\em Acta Arith.}, 42:367--389, 1983.

\bibitem[SV18]{kwise4}
Benny Sudakov and Pedro Vieira.
\newblock Two remarks on eventown and oddtown problems.
\newblock {\em SIAM J. Discrete Math.}, 32(1):280--295, 2018.

\end{thebibliography}

\end{document}